\documentclass{amsart} 

\usepackage{amsmath,amssymb,amsthm} 
\usepackage{graphicx} 
\usepackage{enumerate}
\usepackage{hyperref} 

\newcommand{\qua}{\hskip 0.4em \ignorespaces}
\def\arxiv#1{\relax\ifhmode\unskip\qua\fi
\href{http://arxiv.org/abs/#1}%
{\tt arXiv:\penalty -100\unskip#1}}

\def\MR#1{\relax\ifhmode\unskip\qua\fi
\href{https://mathscinet.ams.org/mathscinet-getitem?mr=#1}{\tt MR#1}}
\def\ZB#1{\relax\ifhmode\unskip\qua\fi
\href{https://zbmath.org/?q=an:#1}{\tt Zbl\:#1}}
\def\xox#1{\csname xx#1\endcsname}

\renewenvironment{thebibliography}[1]{
  \begin{oldthebibliography}{ABC12}\small
    \setlength{\itemsep}{.5ex}
    \setlength{\parskip}{0em}
}
{
  \end{oldthebibliography}
}

\DeclareMathOperator{\rot}{rot}
\DeclareMathOperator{\tb}{tb}

\newcommand{\Z}{\mathbb{Z}}
\newcommand{\Q}{\mathbb{Q}}
\newcommand{\N}{\mathbb{N}}

\newcommand{\Kh}{\operatorname{Kh}}

\newtheoremstyle{thm}{}{}{\itshape}{}{\bfseries}{}{ }{} 
\newtheoremstyle{definition}{}{}{}{}{\bfseries}{}{ }{} 

\theoremstyle{thm}
\newtheorem{Theorem}{Theorem}[section]
\newtheorem{theorem}[Theorem]{Theorem}

\newtheorem{corollary}[Theorem]{Corollary}
\newtheorem{conjecture}[Theorem]{Conjecture}

\theoremstyle{definition}

\newtheorem{rem}[Theorem]{Remark}

\numberwithin{equation}{section}

\begin{document}
\title{Legendrian simple knots not detected by Khovanov homology}

\author{Chun-Sheng Hsueh}
\address{Humboldt-Universit\"at zu Berlin, Rudower Chaussee 25, 12489 Berlin, Germany}
\email{chun-sheng.hsueh@hu-berlin.de}
\email{david.suchodoll@hu-berlin.de} 

\author{Marc Kegel}
\address{Universidad de Sevilla, Dpto.\ de Álgebra,
Avda.\ Reina Mercedes s/n,
41012 Sevilla, Spain}
\email{kegelmarc87@gmail.com}

\author{David Suchodoll}

\author{Annika Thiele}
\address{École Polytechnique Fédérale de Lausanne, UPHESS, Station 8, CH-1015 Lausanne, Switzerland}
\email{annika.thiele@epfl.ch}


\date{\today} 

\begin{abstract}
In this short note, we exhibit a Legendrian simple knot that is not detected by Khovanov homology. This disproves a conjecture of Chernov and Maguire.
\end{abstract}


\makeatletter
\@namedef{subjclassname@2020}{%
 \textup{2020} Mathematics Subject Classification}
\makeatother


\maketitle

\section{A composite example}

Khovanov homology detects\footnote{Here we say that a knot $K$ is \textit{detected} by its Khovanov homology if whenever there is a bigraded isomorphism between the Khovanov homologies $\Kh(K)$ and $\Kh(K')$ of $K$ and another knot $K'$, then $K'$ is isotopic to $K$. There are different versions of Khovanov homology. While the unknot detection is for $\Z$-coefficients (and thus by the universal coefficient theorem holds for all other coefficients), some of the other detection results are currently only proven for $\Q$- or $\Z_2$-coefficients or in reduced Khovanov homology.} the unknot~\cite{KM11}, the trefoils~\cite{BS22},
the figure-eight knot~\cite{BDLLS21}, and the torus knots
$T(2,\pm5)$~\cite{BHS25}; it also detects the non-fibered, twist knot
$K5a1$~\cite{BS25}. On the other hand, the unknot~\cite{EF}, the torus knots, and the figure-eight knot are Legendrian simple\footnote{A knot $K$ is \textit{Legendrian simple} if every isotopy class of a Legendrian realization of $K$ is determined by its classical invariants, the Thurston--Bennequin invariant and the rotation number.}~\cite{EH01}. 
Motivated by these and
related examples, Chernov and Maguire conjectured that every
Legendrian simple knot is detected by Khovanov homology
\cite[Conjecture~3]{CM26}. They supported the conjecture by extensive
computations among prime knots with at most $20$ crossings. In this short note, we present a counterexample to this conjecture.

\begin{theorem}\label{thm:counterexample}
Let $K$ be the connected sum of two figure-eight knots. Then $K$ is Legendrian simple, but $K$ shares the same integral Khovanov homology with infinitely many other knots.
\end{theorem}

\begin{proof}
Etnyre and Honda proved that the figure-eight knot is Legendrian simple and
has a unique non-destabilizable Legendrian representative
\cite[Lemma 5.2, Theorem 5.3, Proposition 5.9]{EH01}. Thus, its Legendrian mountain range has one peak. An's
connected-sum criterion~\cite[Theorem~1]{An16} now implies that $K$, 
the connected sum of two figure-eight knots, is Legendrian simple.

On the other hand, the main result of Watson~\cite{Wat07} shows that $K$ shares the same integral Khovanov homology with infinitely many other knots, see Section~7.2 in~\cite{Wat07}. More precisely, Watson constructs an infinite family of pairwise non-isotopic knots $(K_n)_{n\in\N}$ that have isomorphic integral bigraded Khovanov homologies. The first three knots of this family are $K$ (the connected sum of two figure-eight knots), $K8a16$, and $K12n462$.
\end{proof}

\section{A prime conjectural example}

Next, we present a prime knot that is not detected by its Khovanov homology, but is conjectured to be Legendrian simple. The $13$-crossing knot $K13n1836$ belongs to one of Watson's Kanenobu families~\cite{Wat07}. In fact, using the braid word $[-1,2,-1,-1]$ in Watson's construction as in~\cite[Section 7.1]{Wat07} yields a Kanenobu family with identical integral Khovanov homology groups where the first three knots are $K8a4$, $-K10n18$, and $K13n1836$.\footnote{The isomorphism of their integral bigraded Khovanov homologies can also be verified directly from the computations recorded in KnotInfo~\cite{KnotInfo}.}
As such, $K13n1836$ has the same integral bigraded
Khovanov homology as infinitely many pairwise non-isotopic knots. Mirroring the entire family shows that the same is true for the mirror $-K13n1836$.

On the other hand, there is substantial computational evidence that $K13n1836$ and its mirror are both Legendrian simple. Petkova and Schwartz \cite{PetkovaSchwartz}, cf.\ \cite{LSULegendrianAtlas}, compute the maximal Thurston--Bennequin invariant of any Legendrian representative of $K13n1836$ to be $\overline{\tb}(K13n1836)=-4$ and that there are precisely two maximal-$\tb$ Legendrian representatives, represented by the grids \(L_-=G730\) and \(L_+=G731\), with classical invariants
\[
        \tb(L_\pm)=-4,
        \qquad
        \rot(L_\pm)=\pm1.
\]
Their computation in the next horizontal layer (i.e.\ those grids with $\tb=-5$) contains a representative $S777$
with \(\tb(S777)=-5\) and $\rot(S777)=0$ which has both \(L_+\) and \(L_-\) as parents.

This yields the following partial result towards Legendrian simplicity of \(K13n1836\).

\begin{corollary}\hfill
\begin{enumerate}
    \item Any two Legendrian representatives of \(K13n1836\) obtained by stabilizing
\(L_-\) or \(L_+\) are Legendrian isotopic whenever they have the same
Thurston--Bennequin and rotation numbers.
    \item If every Legendrian representative of \(K13n1836\) with
\(\tb<-4\) destabilizes, then \(K13n1836\) is Legendrian simple.
\end{enumerate}
\end{corollary}

\begin{proof}
The representatives \(L_-\) and \(L_+\) have a common stabilization. Hence,
using that positive and negative stabilizations commute, any stabilizations of
\(L_-\) and \(L_+\) with the same classical invariants are Legendrian isotopic.
This proves~(1).

If every representative with \(\tb<-4\) destabilizes, then every Legendrian
representative destabilizes to one of the two maximal-\(tb\) representatives
\(L_-\) or \(L_+\). Thus~(1) implies that \(\tb\) and \(\rot\) determine the
Legendrian isotopy class, proving~(2).
\end{proof}

\begin{rem}
The mirror $-K13n1836$ exhibits exactly the same phenomenon. Petko\-va and Schwartz~\cite{PetkovaSchwartz}, cf.\ \cite{LSULegendrianAtlas}, compute $\overline{\tb}(-K13n1836)=-6$
and find precisely two maximal $\tb$ representatives $G1661$ and $G1662$, with
\begin{equation*}
    \rot(G1661)=-1,
\qquad
\rot(G1662)=1.
\end{equation*}
Moreover, $S1983$ has $\tb=-7$ and $\rot=0$ and has both $G1661$ and $G1662$ as parents. Thus, $G1661$ and $G1662$ have a common stabilization. Consequently, any stabilizations of these two representatives with the same classical invariants are Legendrian isotopic; and if every Legendrian representative of $-K13n1836$ with $\tb<-6$ destabilizes, then $-K13n1836$ is Legendrian simple.
\end{rem}

\begin{rem}
Since the Seifert genus of $K13n1836$ is two and since the maximal Thurston--Bennequin invariants of $K13n1836$ and its mirror is less than $-1$, Theorem~1.2 in~\cite{Etnyre_neighborhood} implies that every nondestabilizable Legendrian representative satisfies
$$
|\rot(L)|\le 3.
$$
\end{rem}

All these points lead to the following conjecture.

\begin{conjecture}
The knot \(K13n1836\) and its mirror \(-K13n1836\) are Legendrian simple.
\end{conjecture}

\subsection*{Acknowledgements}

We thank Naageswaran Manikandan for useful initial discussions and John Etnyre and Ina Petkova for clarifications about~\cite{Etnyre_neighborhood} and~\cite{PetkovaSchwartz}.

Although the final arguments of this article are noncomputational, the surrounding search relied heavily on SnapPy~\cite{SnapPy}, SageMath~\cite{Sage}, GridPyM~\cite{BC24}, GridLink~\cite{GridLink}, KnotJob~\cite{KnotJob}, KLO~\cite{KLO}, KnotInfo~\cite{KnotInfo}, and the Legendrian knot atlases~\cite{CN_Leg_knot_atlas,PetkovaSchwartz,LSULegendrianAtlas}. We acknowledge that this project would not have been possible without these programs and data, and we thank the developers for making these freely available.

\subsection*{Individual grant support}

CSH and AT were supported by the Berlin Mathematical School, the Berlin Mathematics Research Center MATH+ (EXC-2046/2, project ID: 390685689), funded by the Deutsche Forschungsgemeinschaft (DFG) under Germany’s Excellence Strategy.
CSH was supported by the Claussen-Simon-Stiftung.
MK is supported by a Ram\'on y Cajal grant (RYC2023-043251-I) and by the project PID2024-157173NB-I00 funded by MCIN/AEI/10.13\-039/501100011033, ESF+ and FEDER, EU; and by a VII Plan Propio de Investigación y Transferencia (SOL2025-36103) of the University of Sevilla.

\bibliographystyle{myamsalpha}
\bibliography{main}

@misc{knotinfo,
Author = {Livingston, Charles and Moore, Allison H.},
howpublished = {\url{https://knotinfo.org}},
Title = {Knot{I}nfo: Table of Knot Invariants},
Year = {2026},
}

@article{KM11,
 author = {Kronheimer, P. B. and Mrowka, T. S.},
 title = {Khovanov homology is an unknot-detector},
 fjournal = {Publications Math{\'e}matiques},
 journal = {Publ. Math., Inst. Hautes {\'E}tud. Sci.},
 issn = {0073-8301},
 volume = {113},
 pages = {97--208},
 year = {2011},
 language = {English},
 doi = {10.1007/s10240-010-0030-y},
 url = {hdl.handle.net/1721.1/70474},
 zbMATH = {5963659},
 Zbl = {1241.57017}
}

@article{BS22,
 author = {Baldwin, John A. and Sivek, S.},
 title = {Khovanov homology detects the trefoils},
 fjournal = {Duke Mathematical Journal},
 journal = {Duke Math. J.},
 issn = {0012-7094},
 volume = {171},
 number = {4},
 pages = {885--956},
 year = {2022},
 language = {English},
 doi = {10.1215/00127094-2021-0034},
 zbMATH = {7500566},
 Zbl = {1494.57020}
}

@article{BDLLS21,
 author = {Baldwin, John A. and Dowlin, Nathan and Levine, Adam Simon and Lidman, Tye and Sazdanovic, Radmila},
 title = {Khovanov homology detects the figure-eight knot},
 fjournal = {Bulletin of the London Mathematical Society},
 journal = {Bull. Lond. Math. Soc.},
 issn = {0024-6093},
 volume = {53},
 number = {3},
 pages = {871--876},
 year = {2021},
 language = {English},
 doi = {10.1112/blms.12467},
 zbMATH = {7381916},
 Zbl = {1470.57025}
}

@article{BHS25,
 author = {Baldwin, John A. and Hu, Ying and Sivek, Steven},
 title = {Khovanov homology and the cinquefoil},
 fjournal = {Journal of the European Mathematical Society (JEMS)},
 journal = {J. Eur. Math. Soc. (JEMS)},
 issn = {1435-9855},
 volume = {27},
 number = {6},
 pages = {2443--2465},
 year = {2025},
 language = {English},
 doi = {10.4171/JEMS/1415},
 zbMATH = {8030707},
 Zbl = {1576.57004}
}

@article{BS25,
 author = {Baldwin, John A. and Sivek, Steven},
 title = {Floer homology and non-fibered knot detection},
 fjournal = {Forum of Mathematics, Pi},
 journal = {Forum Math. Pi},
 issn = {2050-5086},
 volume = {13},
 pages = {65},
 note = {Id/No e1},
 year = {2025},
 language = {English},
 doi = {10.1017/fmp.2024.28},
 zbMATH = {7979662},
 Zbl = {1570.57011}
}

@article{EH01,
 author = {Etnyre, John B. and Honda, Ko},
 title = {Knots and contact geometry {I}: {Torus} knots and the figure eight knot},
 fjournal = {The Journal of Symplectic Geometry},
 journal = {J. Symplectic Geom.},
 issn = {1527-5256},
 volume = {1},
 number = {1},
 pages = {63--120},
 year = {2001},
 language = {English},
 doi = {10.4310/JSG.2001.v1.n1.a3},
 zbMATH = {1911428},
 Zbl = {1037.57021}
}

@misc{CM26,
      title={Conjectures on the {K}hovanov Homology of Torus Knots, Twist Knots, and {L}egendrian Simple Knots}, 
      author={Vladimir Chernov and Ryan Maguire},
      year={2026},
      arxiv={2205.11430}, 
      howpublished = {Preprint},
}

@misc{Etnyre_neighborhood,
      title={Neighborhoods of transverse knots and destabilizations}, 
      author={John B. Etnyre},
      year={2026},
      arxiv={2512.15651}, 
      howpublished = {Preprint},
}

@article{An16,
 author = {An, Byung Hee},
 title = {A criterion for the {Legendrian} simplicity of the connected sum},
 fjournal = {Topology and its Applications},
 journal = {Topology Appl.},
 issn = {0166-8641},
 volume = {204},
 pages = {175--184},
 year = {2016},
 language = {English},
 doi = {10.1016/j.topol.2016.03.011},
 zbMATH = {6571455},
 Zbl = {1337.57007}
}

@article{Wat07,
 author = {Watson, Liam},
 title = {Knots with identical {Khovanov} homology},
 fjournal = {Algebraic \& Geometric Topology},
 journal = {Algebr. Geom. Topol.},
 issn = {1472-2747},
 volume = {7},
 pages = {1389--1407},
 year = {2007},
 language = {English},
 doi = {10.2140/agt.2007.7.1389},
 zbMATH = {5220917},
 Zbl = {1137.57020}
}

@article{PetkovaSchwartz,
 author = {Petkova, Ina and Schwartz, Noah},
 title = {A {Legendrian} knot atlas for knots of arc index 10},
 fjournal = {Experimental Mathematics},
 journal = {Exp. Math.},
 issn = {1058-6458},
 volume = {35},
 number = {1},
 pages = {111--259},
 year = {2026},
 language = {English},
 doi = {10.1080/10586458.2024.2430715},
 zbMATH = {8176797},
 Zbl = {08176797}
}

@article{CN_Leg_knot_atlas,
 author = {Chongchitmate, Wutichai and Ng, Lenhard},
 title = {An atlas of {Legendrian} knots},
 fjournal = {Experimental Mathematics},
 journal = {Exp. Math.},
 issn = {1058-6458},
 volume = {22},
 number = {1},
 pages = {26--37},
 year = {2013},
 language = {English},
 doi = {10.1080/10586458.2013.750221},
 zbMATH = {6180370},
 Zbl = {1267.57004}
}

@misc {LSULegendrianAtlas,
    AUTHOR = {Bhattacharyya, Nilangshu and Cox, Cyrus and Murray, Justin
              and Pandikkadan, Adithyan and Vela-Vick, Shea and Wu, Angela},
     TITLE = {Legendrian Knot Atlas},
     howpublished = {\url{https://www.math.lsu.edu/~knotatlas/legendrian/}},
}

@misc {SnapPy,
    AUTHOR = {Culler, Marc and Dunfield, Nathan M. and Goerner, Matthias
              and Weeks, Jeffrey R.},
     TITLE = {Snap{P}y, a computer program for studying the geometry
              and topology of 3-manifolds},
    howpublished = {\url{https://snappy.computop.org/}},
}

@misc {Sage,
    AUTHOR = {{The Sage Developers}},
     TITLE = {Sage{M}ath, the Sage Mathematics Software System},
         howpublished = {\url{https://www.sagemath.org/}},
}

@article{BC24,
 author = {Barbensi, Agnese and Celoria, Daniele},
 title = {{GridPyM}: a {Python} module to handle grid diagrams},
 fjournal = {The Journal of Software for Algebra and Geometry},
 journal = {J. Softw. Algebra Geom.},
 issn = {1948-7916},
 volume = {14},
 pages = {31--39},
 year = {2024},
 language = {English},
 doi = {10.2140/jsag.2024.14.31},
 zbMATH = {7878541},
 Zbl = {1544.57002}
}

@misc {GridLink,
    AUTHOR = {Culler, Marc},
     TITLE = {Grid{L}ink, A tool for manipulating rectangular link diagrams},
     howpublished = {\url{https://homepages.math.uic.edu/~culler/gridlink/}},
}

@misc {KnotJob,
    AUTHOR = {Sch{\"u}tz, Dirk},
     TITLE = {Knot{J}ob},
     howpublished = {\url{https://www.maths.dur.ac.uk/users/dirk.schuetz/knotjob.html}},
}

@misc {KLO,
    AUTHOR = {Swenton, Frank},
     TITLE = {{KLO} (Knot-Like Objects)},
     howpublished = {\url{https://community.middlebury.edu/~mathanimations/klo/}},
}

@article{EF,
 author = {Eliashberg, Yakov and Fraser, Maia},
 title = {Topologically trivial {Legendrian} knots},
 fjournal = {The Journal of Symplectic Geometry},
 journal = {J. Symplectic Geom.},
 issn = {1527-5256},
 volume = {7},
 number = {2},
 pages = {77--127},
 year = {2009},
 language = {English},
 doi = {10.4310/JSG.2009.v7.n2.a4},
 zbMATH = {5573151},
 Zbl = {1179.57040}
}
\end{document}